\documentclass[11pt, epsfig]{article}
\usepackage{enumerate}
\usepackage{epsfig, amsmath, amssymb, amsthm, times}
\numberwithin{equation}{section}

\def\N{{\mathbb N}}
\def\be{\begin{equation}}
\def\ee{\end{equation}}
\def\lbl{\label}

\def\c{\underline{c}}
\def\b{\underline{b}}
\def\a{\underline{a}}
\def\d{\underline{d}}
\def\x{\underline{x}}
\def\var{\mathbb{V}\mathrm{ar}}
\def\cov{\mathbb{C}\mathrm{ov}}

\def\Q{\mathcal Q}

\newtheorem{thm}{Theorem}[section]

\newtheorem{prop}[thm]{Proposition}

\newtheorem{rem}[thm]{Remark}

\def\E{{\mathbb E}}
\def\P{{\mathbb P}}
\def\eps{\epsilon}

\def\R{\mathbb{R}}

\title{Asymptotic normality \\ through \\ factorial cumulants and
  partitions identities\thanks{KB, FLB, and JW were  partially supported  by Spanish Ministry of
 Education and Science Grants MTM2007-65909   MTM2010-16949. PH was
 partially supported  by NSA Grant $\#$H98230-09-1-0062 and   Simons
 Foundation grant \#208766. GR and JW were partially supported by NSF
 Grant DMS0840695 and NIH Grant R01CA152158.
} }
\author{\small{Konstancja Bobecka}\thanks{1 Wydzia{\l} Matematyki i Nauk
      Informacyjnych, Politechnika Warszawska, Warszawa, Poland
      e-mail: bobecka@mini.pw.edu.pl}
\hspace{1mm},\hspace{1mm} Pawe{\l} Hitczenko\thanks{2 Department of
  Mathematics, Drexel University, Philadelphia, USA, e-mail:
  phitczenko@math.drexel.edu}
\hspace{1mm},\hspace{1mm} Fernando L\'opez-Bl\'azquez\thanks{Facultad
  de Matem\'aticas Universidad de Sevilla, Sevilla, Spain, e-mail:
  lopez@us.es}
\hspace{1mm},\hspace{1mm} Grzegorz Rempa\l a\thanks{Department of
  Biostatistics, Medical College of Georgia, Augusta, USA, e-mail:
  grempala@mcg.edu}
\hspace{1.5mm},\hspace{1mm} Jacek Weso\l owski\thanks{Wydzia{\l} Matematyki i Nauk Informacyjnych, Politechnika Warszawska, Warszawa, Poland, e-mail: wesolo@mini.pw.edu.pl}}
\date{}
\begin{document}
\maketitle
\begin{abstract}
In the paper we develop an approach to asymptotic normality through
factorial cumulants. Factorial cumulants arise in the same manner from
factorial moments, as do (ordinary) cumulants from (ordinary) moments. Another tool we exploit is a new identity for "moments" of partitions of numbers. The general limiting result is then used to (re--)derive asymptotic normality for several models including classical discrete distributions, occupancy problems in some generalized allocation schemes and two models related to negative multinomial distribution.
\end{abstract}

\section{Introduction}
Convergence to normal law is one of the  most important phenomena of probability. As a consequence, a number of general methods, often based on transforms of the underlying sequence have been developed as techniques for establishing such convergence. One of these methods, called the method of moments,  rests on the fact that the standard normal random variable is uniquely determined by its moments  and  that for  such random variable  $X$ if $(X_n)$ is a sequence of random variables having all moments and $\E\,X_n^k\to\E\,X^k$ for all $k=1,2,\ldots$, then $X_n\stackrel{d}{\to}X$: see, e.g., \cite[Theorem~ 2.22]{v}  or \cite[Theorem~30.2]{b}. Here and throughout the paper we use \lq\lq$\stackrel d\to$\rq\rq\ to denote the convergence in distribution.

Since moments are not always convenient to work with, one can often
use some other characteristics of random variables to establish the
convergence to the normal law. For example,  in one  classical situation we  consider a sequence of  cumulants (we recall the definition in the next section) rather than moments. On the one hand,
since the $k$th cumulant is a continuous function of the first $k$ moments, to prove that $X_n\stackrel{d}{\to}X$ instead of convergence of moments one can use convergence of cumulants.
On the other hand,  all cumulants of the standard normal distribution
are zero except of the cumulant of order 2 which equals 1. This  often
makes it much easier to establish the convergence of cumulants of
$(X_n)$ to the cumulants of the standard normal random variable.  We
refer the reader to e.g. \cite[Section~6.1]{jlr} for a  more detailed discussion.

In this paper we develop  an approach to asymptotic normality that is
based on factorial cumulants.  They will be discussed in the  next
section. Here we just indicate that factorial cumulants arise in the
same manner from factorial moments, as do (ordinary) cumulants from (ordinary) moments.
The motivation for our work is the fact that quite often one encounters  situations in which properties of random variables are naturally expressed through factorial (rather than ordinary) moments. As we will see below, this is the case, for  instance, when random variables under consideration are  sums of indicator variables.

In developing our approach  we  first provide a simple and yet quite
general sufficient condition for the Central Limit  Theorem (CLT) in terms of factorial cumulants (see Proposition~\ref{gen0} below). Further, as we will see in Theorem~\ref{gen}, we show that the validity of this condition can be verified by controlling the asymptotic behavior of factorial moments.
 This limiting result will be  then used in Section~\ref{sec:appl} to
 (re--)derive asymptotic normality for several models including
 classical discrete distributions, occupancy problems in some
 generalized allocation schemes (GAS),  and two models related to a
 negative multinomial distribution; they are examples of what may be
 called  generalized inverse allocation schemes (GIAS). Gas were
 introduced in \cite{k} and we refer the reader to chapters in books
 \cite{k1,k2} by the same author for more details, properties, and
 further references. The term \lq\lq generalized inverse allocation
 schemes\rq\rq\ does not seem to be commonly used in the literature;
 in our terminology the word \lq\lq inverse\rq\rq\ refers to  inverse
 sampling, a sampling method   proposed in \cite{h}. A number of
 distributions with \lq\lq inverse\rq\rq\ in their names reflecting
 the inverse sampling are discussed in the first (\cite{jkk}) and the
 fourth (\cite{jkb}) volumes of the Wiley Series in Probability and Statistics.

 We believe that our approach may turn out to be useful in other
 situations when the factorial moments are natural and convenient
 quantities to work with.  We wish to mention, however,  that although
 several of our summation schemes are, in fact,  GAS or GIAS we do not
 have a general statement that would give  reasonably general
 sufficient conditions under which a GAS or a GIAS exhibits asymptotic
 normality. It may be and interesting question, worth further investigation.

Aside from utilizing factorial cumulants, another technical tool  we
exploit is an identity for \lq\lq moments\rq\rq\  of partitions of
natural numbers (see Proposition~\ref{partid}). As far as we know this identity is new,
and may be of independent interest to the combinatorics community. As of now, however, we do not have any combinatorial interpretation, either for its validity or for its proof.

\section{Factorial cumulants}\label{sec:fac_cum}
Let $X$ be a random variable with the Laplace transform
$$
\phi_X(t)=\E\:e^{tX}=\sum_{k=0}^{\infty}\:\mu_k\frac{t^k}{k!}
$$
and the cumulant transform
$$
\psi_X(t)=\log(\phi_X(t))=\sum_{k=0}^{\infty}\:\kappa_k\frac{t^k}{k!}\:.
$$
Then $\mu_k=\E\,X^k$ and $\kappa_k$ are, respectively, the $k$th moment and the $k$th cumulant of $X$, $k=0,1,\ldots$.

One can view the sequence $\underline{\kappa}=(\kappa_k)$ as obtained by a transformation
$f=(f_k)$ of the sequence $\underline{\mu}=(\mu_k)$, that is,
$\underline{\kappa}=f(\underline{\mu})$, where the $f_k$ are
defined recursively by $f_1(\x)=x_1$ and
\be\label{fid}f_k(\x)=x_k-\sum_{j=1}^{k-1}\:\binom{k-1}{j-1}f_j(\x)x_{k-j},\qquad
k>1. \ee

The Laplace transform can be also expanded in the form
\be\label{famo}
\phi_X(t)=\sum_{k=0}^{\infty}\:\nu_k\frac{\left(e^t-1\right)^k}{k!},
\ee
where $\nu_k$ is the $k$th factorial moment of $X$, that is, 
\[\nu_0=1\quad\mbox{and}\quad
\nu_k=\E\,(X)_k=\E\,X(X-1)\ldots(X-k+1),\quad k=1,\ldots.\] 
Here and below we use the Pochhammer symbol  $(x)_k$ for the falling factorial $x(x-1)\dots(x-(k-1))$.

In analogy to \eqref{fid} one can view the sequence $\underline{\mu}$ as  obtained by a transformation
$g=(g_k)$ of the sequence $\underline{\nu}=(\nu_k)$, that is,
$\underline{\mu}=g(\underline{\nu})$, where the $g_k$ are
defined  by
\be\label{gid}
g_k(\x)=\sum_{j=1}^k\:S_2(k,j)x_j,\qquad k\ge 1,
\ee
where $S_2(k,j)$ are the Stirling numbers of the second kind defined by the identity
\[y^k=\sum_{j=1}^k\:S_2(k,j)(y)_j,\] 
holding for any $y\in\R$ (see e.g. \cite[(6.10)]{gkp}).

Following this probabilistic terminology for any
sequence of real numbers $\a=(a_k)$ one can define its cumulant and factorial
sequences, $\b=(b_k)$ and $\c=(c_k)$, respectively,   by
\be\lbl{exp_form}\sum_{k=0}^{\infty}\,a_k\frac{t^k}{k!}=\exp\left\{\sum_{k=0}^{\infty}\,b_k\frac{t^k}{k!}\right\}=\sum_{k=0}^{\infty}\,c_k\frac{\left(e^t-1\right)^k}{k!}.\ee
The first relation is known in combinatorics by the name \lq
exponential formula\rq,  and its combinatorial interpretation when both $(a_n)$ and $(b_n)$ are non--negative integers may be found, for example, in \cite[Section~5.1]{s_ec}.

Note that if $\a$, $\b$ and $\c$ are related by \eqref{exp_form} then $\b=f(\a)$ and  $\a=g(\c)$, where $f$ and $g$ are given by \eqref{fid} and \eqref{gid}, respectively.

Let the sequence $\d=(d_k)$ be defined by
$$
\exp\left\{\sum_{k=0}^{\infty}\:d_k\frac{\left(e^t-1\right)^k}{k!}\right\}=\sum_{k=0}^{\infty}\:c_k\frac{\left(e^t-1\right)^k}{k!}.
$$
Then,  regarding $e^t-1$ as a new variable we see that
$\d=f(\c)$. Since $\c$ is a factorial sequence for $\a$ and $\d$ is a
cumulant sequence for $\c$,  we call $\d=f(\c)$ the {\em factorial cumulant sequence} for $\a$.

Observe that, by \eqref{exp_form},
$$
\sum_{k=0}^{\infty}\:d_k\frac{\left(e^t-1\right)^k}{k!}=\sum_{k=0}^{\infty}\, b_k\frac{t^k}{k!}
$$
and thus $\b=g(\d)=g(f(\c))$. That is,
\be\label{bexp}
b_k=\sum_{j=1}^k\:S_2(k,j)f_j(\c)\;,\quad k=1,2,\ldots \ee
This observation is useful for proving convergence in law to the standard normal variable.

\begin{prop}\label{gen0}
Let $(S_n)$ be a sequence of random variables having all moments.  Assume that
\be\label{war1}
\var S_n\stackrel{n\to\infty}{\longrightarrow} \infty
\ee
and
\be\label{war2}
\frac{\E\,S_n}{\var^{\frac{3}{2}}S_n}\stackrel{n\to\infty}{\longrightarrow}0\:.
\ee
For any $n=1,2,\ldots$, let $\c_n=(c_{k,n})_{k=1,\ldots}$ be the sequence of factorial moments of $S_n$, that is, $c_{k,n}=\E\,(S_n)_k$, $k=1,2,\ldots$, and let $f_{J,n}=f_J(\c_n)$ (where $f_J$ is defined by \eqref{fid}) be the $J$th factorial cumulant of $S_n$, $J=1,2,\ldots$. Assume that
\be\label{fJc}
\frac{f_{J,n}}{\var^{\frac{J}{2}}S_n}\stackrel{n\to\infty}{\longrightarrow}0\quad \mbox{for }\;J\ge 3.
\ee
Then
\be\label{clt}
U_n=\frac{S_n-\E\,S_n}{\sqrt{\var\,S_n}}\stackrel{d}{\to}{\cal N}(0,1)\;.
\ee
\end{prop}
\begin{proof}
We will use the cumulant convergence theorem (see e.g. \cite[Theorem~6.14]{jlr}).
 Let $\kappa_{J,n}$ denote the $J$th cumulant of $U_n$ and recall  that all cumulants of the standard normal distribution are zero except for the cumulant of order 2, which is 1. It is obvious that $\kappa_{1,n}=0$ and $\kappa_{2,n}=1$.  Therefore, to prove \eqref{clt} it suffices to show that $\kappa_{J,n}\to 0$ for  $J\ge 3$.
 By \eqref{bexp}, 
$$
\kappa_{J,n}=\frac{\sum_{j=1}^J\:S_2(J,j)f_{j,n}}{\var^{\frac{J}{2}}S_n}.
$$
Fix arbitrary $J\ge 3$. To prove that $\kappa_{J,n}\to 0$ it suffices to show that
\be\label{fjn}
\frac{f_{j,n}}{\var^{\frac{J}{2}}S_n}\to 0\quad\mbox{for all}\;j=1,2,\ldots,J.
\ee
Note first that by \eqref{fid}
$$
f_{1,n}=\E\,S_n\quad\mbox{and}\quad f_{2,n}=\var(S_n)-\E\,S_n
$$
Therefore the assumptions \eqref{war1} and \eqref{war2} imply \eqref{fjn} for $j=1,2$.

If $j\in\{3,\ldots,J-1\}$, we write
$$
\frac{f_{j,n}}{\var^{\frac{J}{2}}S_n}=\frac{f_{j,n}}{\var^{\frac{j}{2}}S_n}\:\frac{1}{\var^{\frac{J-j}{2}}S_n}.
$$
By \eqref{fJc} the first factor tends to zero and by \eqref{war1} the second factor tends to zero as well.

Finally, for $j=J$ the conditions \eqref{fjn} and \eqref{fJc} are identical.
\end{proof}

The above result is particularly useful when the factorial moments of $S_n$ are available in a nice form. We will now describe  a general situation when this happens.

For any numbers $\delta_j$, $j=1,\ldots,N$, assuming
values $0$ or $1$ we have
$$x^{\sum_{j=1}^N\:\delta_j}=\sum_{m=0}^{\sum_{j=1}^N\:\delta_j}\binom{\sum_{j=1}^N\:\delta_j}{m}(x-1)^m
=1+\sum_{m=1}^N(x-1)^m\sum_{1\le j_1<\ldots<j_m\le N}\:\delta_{j_1}\ldots\delta_{j_m}\;.$$
Therefore, if $(\eps_1,\ldots,\eps_N)$ is a random vector valued in
$\{0,1\}^N$ and $S=\sum_{i=1}^N\;\eps_i$ then
$$
\E\,e^{tS}=1+\sum_{m=1}^{\infty}\:\left(e^{t}-1\right)^m\sum_{1\le
j_1<\ldots<j_m\le
N}\:\P(\eps_{j_1}=\eps_{j_2}=\ldots=\eps_{j_m}=1)\;.$$
Comparing this formula with \eqref{famo},  we conclude that factorial moments of $S$ have the form
\be\label{ce} \E(S)_k=k!\sum_{1\le
j_1<\ldots<j_k\le
N}\:\P(\eps_{j_1}=\eps_{j_2}=\ldots=\eps_{j_k}=1)=:c_k\;,\quad
k=1,2,\ldots. \ee
If, in addition, the random variables $(\eps_1,\ldots,\eps_N)$ are exchangeable,
then the above formula simplifies to
\be\label{ce_simp}
\E(S)_k=(N)_k\:\P(\eps_{1}=\ldots=\eps_k=1)=:c_k\;,\quad
k=1,2,\ldots. \ee

As we will see in Section~\ref{sec:appl}, our sufficient condition for asymptotic normality will work well for several set--ups falling within such a scheme.  This will be preceded by a derivation of new identities for integer
partitions,  which will give a major enhancement of the tools we will use to prove limit theorems.

\section{Partition identities}\label{sec:part_id}
Recall  that if $\b=(b_n)$ is a cumulant
sequence for a sequence of numbers $\a=(a_n)$, that is,  $\b=f(\a)$ with $f$ given by \eqref{fid}, then  for  $J\ge1$
\be\lbl{cum_via_mom}
b_J=\sum_{\pi\subset J}\:D_{\pi} \prod_{i=1}^J\:a_i^{m_i},\quad \mbox{
where}\quad
D_{\pi}=\frac{(-1)^{\sum\limits_{i=1}^Jm_i-1}J!}{\prod\limits_{i=1}^J\:(i!)^{m_i}\sum\limits_{i=1}^J\:m_i}
\binom{\sum\limits_{i=1}^J\:m_i}{m_1,\ldots,m_J}\;,
\ee
and where the sum is  over all partitions $\pi$ of a positive integer
$J$, i.e.,  over all
vectors $\pi=(m_1,\ldots,m_J)$ with non--negative integer components
which satisfy $\sum_{i=1}^Jim_i=J$ (for a proof, see,  e.g.,  \cite[Section~3.14]{ks}).

Note that for $J\ge 2$
\be\label{for2}
\sum_{\pi\subset
J}\:D_{\pi}=0\;.
\ee
This follows from the fact that all the cumulants,  except for the first
one, of the  constant random variable $X=1$ a.s. are zero.

For $\pi=(m_1,\ldots,m_J)$ we denote
$H_{\pi}(s)=\sum_{i=1}^J\:i^sm_i$, $s=0,1,2\ldots$.
The main result of this section is the identity which extends considerably \eqref{for2}.

\begin{prop}\label{partid}
Assume  $J\ge 2$. Let $I\ge 1$ and $s_i\ge 1$, $i=1,\ldots,I$,
be such that
$$
\sum_{i=1}^I\:s_i\le J+I-2.
$$
Then \be\label{for7} \sum_{\pi\subset J}\:D_{\pi}
\prod_{i=1}^I\:H_{\pi}(s_i)=0. \ee
\end{prop}

\begin{proof}
We use induction with respect to $J$. Note that if $J=2$ than $I$
may be arbitrary but $s_i=1$ for all $i=1,\ldots,I$. Thus the identity
\eqref{for7} follows from \eqref{for2} since $H_{\pi}(1)=J$ for
any $J$ and any $\pi\subset J$.

Now  assume that the result holds true for
$J=2,\ldots,K-1$, and consider the case of $J=K$. That is, we want
to study
$$
\sum_{\pi\subset K}\:D_{\pi}\prod_{i=1}^I\:H_{\pi}(s_i)
$$
under the condition $\sum_{i=1}^I\:s_i\le K+I-2$.

Let us introduce functions $g_i$, $i=1,\ldots,K$, by letting
$$
g_i(m_1,\ldots,m_K)=(\tilde{m}_1,\ldots,\tilde{m}_{K-1})=\left\{\begin{array}{ll}
(m_1,\ldots,m_{i-1}+1,m_i-1,\ldots,m_{K-1}) & \mbox{if }\;i\ne
1,K, \\
(m_1-1,m_2,\ldots,m_{K-1}) & \mbox{if }\:i=1, \\
(m_1,\ldots,m_{K-2},m_{K-1}+1) & \mbox{if }\;i=K\;.
\end{array}\right.
$$
Note that 
\[g_i:\{\pi\subset K:\:m_i\ge 1\}\to\{\tilde{\pi}\subset
(K-1):\:\tilde{m}_{i-1}\ge 1\},\quad i=1,\ldots,K,\]
 are bijections.
Here for consistency we assume $\tilde{m}_0=1$.

Observe that for any $s$, any $\pi\subset K$ such that $m_i\ge 1$,
and for $\tilde{\pi}=g_i(\pi)\subset (K-1)$ we have
$$
H_{\tilde{\pi}}(s)=H_{\pi}(s)-i^s+(i-1)^s=H_{\pi}(s)-1-A_s(i-1)\;,
$$
where 
\[A_s(i-1)=\sum_{k=1}^{s-1}\binom{s}{k}(i-1)^k\] 
is a
polynomial of the degree $s-1$ in the variable $i-1$ with
constant term equal to zero. In particular, for $i=1$ we have
$H_{\tilde{\pi}}(s)=H_{\pi}(s)-1$.
Therefore, expanding $H_{\pi}(s_1)$ we obtain
$$
\sum_{\pi\subset
K}\:D_{\pi}\prod_{i=1}^I\:H_{\pi}(s_i)=\sum_{i=1}^K\:i^{s_1}\sum_{\substack{\pi\subset
K \\ m_i\ge
1}}\:m_iD_{\pi}\prod_{j=2}^I\:\left[H_{\tilde{\pi}}(s_j)+1+A_{s_j}(i-1)\right].
$$
Note that if $m_i\ge 1$ then
$$
\frac{1}{K!}i
m_iD_{\pi}=\left\{\begin{array}{ll}\frac{(-1)^{M_{\pi}-1}M_{\pi}!}
{M_{\pi}(m_1-1)!\prod\limits_{k=2}^{K-1}m_k!(k!)^{m_k}}=-\frac{1}{(K-1)!}\sum_{j=1}^{K-1}\:D_{\tilde{\pi}}\tilde{m}_j\quad
& \mbox{if }\;i=1, \\
\; & \; \\ \frac{(-1)^{M_{\pi}-1}M_{\pi}!}
{M_{\pi}(m_i-1)!(i-1)!(i!)^{m_i-1}\prod\limits_{\substack{k=2 \\
k\ne
i}}^Km_k!(k!)^{m_k}}=\frac{1}{(K-1)!}D_{\tilde{\pi}}\tilde{m}_{i-1}\quad
& \mbox{if }\;i=2,\ldots,K,
\end{array}\right.
$$
where $\tilde{\pi}=(\tilde{m}_1,\ldots,\tilde{m}_{K-1})=g_i(\pi)$,
respectively.
Therefore, \be\label{rozbicie} \frac{1}{K}\sum_{\pi\subset
K}\:D_{\pi}\prod_{i=1}^I\:H_{\pi}(s_i)=
-\sum_{i=1}^{K-1}\sum_{\tilde{\pi}\subset(K-1)}\:D_{\tilde{\pi}}\tilde{m}_i\prod_{j=2}^I\:\left[H_{\tilde{\pi}}(s_j)+1\right]
\ee $$+\sum_{i=2}^K\sum_{\substack{\tilde{\pi}\subset(K-1)\\
\tilde{m}_{i-1}\ge
1}}\:D_{\tilde{\pi}}\tilde{m}_{i-1}i^{s_1-1}\prod_{j=2}^I\:\left[H_{\tilde{\pi}}(s_j)+1+A_{s_j}(i-1)\right].
$$
The second term in the above expression can be written as
\be\label{drugi}
\sum_{i=1}^{K-1}\sum_{\tilde{\pi}\subset(K-1)}\:D_{\tilde{\pi}}\tilde{m}_i(i+1)^{s_1-1}
\prod_{j=2}^I\:\left[H_{\tilde{\pi}}(s_j)+1+A_{s_j}(i)\right]\;.
\ee
Note that
$$
(i+1)^{s_1-1}
\prod_{j=2}^I\:\left[H_{\tilde{\pi}}(s_j)+1+A_{s_j}(i)\right]$$$$=\sum_{r=0}^{s_1-1}\binom{s_1-1}{r}
\sum_{M=0}^{I-1}\sum_{2\le u_1<\ldots<u_M\le
I}\:i^r\prod_{h=1}^M\:(H_{\tilde{\pi}}(s_{u_h})+1)\prod_{\substack{2\le j\le I\\
j\not\in\{u_1,\ldots,u_M\}}}A_{s_j}(i)\;.
$$
The term with $r=0$ and $M=I-1$ in the above expression is
$\prod_{j=2}^I\:\left[H_{\tilde{\pi}}(s_j)+1\right]$, so this
term of the sum \eqref{drugi} cancels with the first term of
\eqref{rozbicie}.

Hence, we only need to show that for any
$r\in\{1,\ldots,s_1-1\}$, any $M\in\{0,\ldots,I-1\}$, and any $2\le
u_1<\ldots<u_M\le I$
\be\label{ssuma}\sum_{\tilde{\pi}\subset(K-1)}D_{\tilde{\pi}}\prod_{h=1}^M\:(H_{\tilde{\pi}}(s_{u_h})+1)
\sum_{i=1}^{K-1}\tilde{m}_ii^r\prod_{\substack{2\le j\le I\\
j\not\in\{u_1,\ldots,u_M\}}}A_{s_j}(i)=0. \ee  Observe
that the
expression $$\sum_{i=1}^{K-1}\tilde{m}_ii^r\prod_{\substack{2\le j\le I\\
j\not\in\{u_1,\ldots,u_M\}}}A_{s_j}(i)$$ is a linear combination of
$H_{\tilde{\pi}}$ functions with the largest value of an argument equal to  $$\sum_{\substack{2\le j\le I\\
j\not\in\{u_1,\ldots,u_M\}}}(s_j-1)+r\;.$$
Therefore the left-hand side of \eqref{ssuma} is a respective linear combination of
terms of the form \be\label{term}
\sum_{\tilde{\pi}\subset(K-1)}D_{\tilde{\pi}}\prod_{w=1}^W\:H_{\tilde{\pi}}(t_w),
\ee
where $$\sum_{w=1}^Wt_w\le \sum_{j=1}^Ms_{u_j}-(M-W+1)+\sum_{\substack{2\le j\le I\\
j\not\in\{u_1,\ldots,u_M\}}}(s_j-1)+s_1-1\le
\sum_{i=1}^I\:s_i-(M-W+1)-(I-1-M)-1\;.$$ But we assumed that
$\sum_{i=1}^Is_i\le K+I-2$. Therefore $$\sum_{w=1}^Wt_w\le
K+I-2-(M-W+1)-(I-1-M)-1=(K-1)+W-2\;.$$ Now, by the inductive
assumption it follows that any term of the form \eqref{term} is
zero, thus proving \eqref{ssuma}.
\end{proof}

Note that in a similar way one can prove that \eqref{for7} is no
longer true when
$$
\sum_{i=1}^I\:s_i= J+I-1\;.
$$

\textbf{Remark}
Richard Stanley \cite{s}  provided the following  combinatorial
 description of the left--hand side of \eqref{for7}. Put  \[F_n(x) = \sum_k S_2(n,k)x^k,
\] and let $(s_i)$ be a sequence of positive integers. Then the left--hand side of \eqref{for7} is a coefficient of $x^J/J!$   in
\[
  \sum_{\cal P} (-1)^{|{\cal P}|-1} (|{\cal P}|-1)! \prod_{B\in{\cal P}} F_{\sigma B}(x),
\]
where the sum ranges over all partitions ${\cal P}$ of a set  
$\{1,...,I\}$  into $|{\cal P}|$  non--empty pairwise disjoint subsets, and where for any such subset $B\in{\cal P}$
\[
  \sigma B = \sum_{i\in B} s_i.
\]
In the simplest case when $I=1$   for any postive integer $s_1$, the left-hand side of equation
 \eqref{for7}  is equal to $J!S_2(s_1,J)$. Since this is the number of surjective maps
 from an $s_1$-element set to a $J$-element set, it must be 0 for
 $s_1<J$,  which is exactly what Proposition~\ref{partid} asserts.
 It is not clear to us how easy it would be to show that \eqref{for7} holds for the larger values of $I$.

\section{Central Limit Theorem: general set--up}
\label{sec:gen_set}
To illustrate and motivate how our approach is intended to work,
consider a sequence $(S_n)$ of Poisson random variables, where
$S_n\sim {\cal P}oisson(\lambda_n)$. As is well known,  if  $\lambda_n\to\infty$ then $U_n=\frac{S_n-\E\:S_n}{\sqrt{\var S_n}}$  converges in distribution to ${\cal N}(0,1)$. To see how it follows from our approach, note that  $\E\,S_n=\var S_n=\lambda_n$ and therefore the assumptions \eqref{war1} and \eqref{war2} of Proposition \ref{gen0} are trivially satisfied.
Moreover,  the factorial moments of $S_n$ are
$c_{i,n}=\lambda_n^i$. Consequently,
$\prod_{i=1}^J\:c_{i,n}^{m_i}=\lambda_n^J$ for any partition
$\pi=(m_1,\ldots,m_J)$ of $J\in \N$,  and thus
$$
f_{J,n}=f_J(\c_n)=\sum_{\pi\subset J}\:D_{\pi}\prod_{i=1}^J\:c_{i,n}^{m_i}=\lambda_n^J\sum_{\pi\subset J}\:D_{\pi}\,.
$$
It now follows from the simplest case \eqref{for2}  of our partition identity  that $f_{J,n}=0$ as long as $J\ge 2$. Hence the assumption \eqref{fJc} of Proposition \ref{gen0} is also trivially satisfied and we conclude the asymptotic normality of  $(U_n)$.

The key in the above argument was, of course,  the very simple form of
the factorial moments $c_{i,n}$ of $S_n$,   which resulted in
factorization of the products of $c_{i,n}^{m_i}$ in the expression for
$f_{J,n}$.   It is not unreasonable, however, to expect that if the
expression for moments does not depart too much from the form it took
for the Poisson variable, then with the full strength of
Proposition~\ref{partid} one might be able to prove the CLT. This is
the essence of condition~\eqref{ln} in the next result. This
condition,  when combined with the  extension of \eqref{for2} given in
Proposition~\ref{partid}, allows us to greatly refine  Proposition \ref{gen0} towards a possible use in schemes of summation of indicators, as  was suggested in the final part of Section~\ref{sec:fac_cum}.

\begin{thm}\label{gen}
Let $(S_n)$ be a sequence of random variables with factorial moments
$c_{i,n}$, $i,n=1,2,\ldots$. Assume that \eqref{war1} and \eqref{war2}
are satisfied and suppose that $c_{i,n}$ can be decomposed into the form
\be\label{ln}
c_{i,n}=L_n^i\exp\left(\sum_{j\ge 1}
\,\frac{Q_{j+1}^{(n)}(i)}{jn^j}\right),\quad i,n=1,2,\ldots,
\ee
where $(L_n)$ is a sequence of real numbers and $Q_j^{(n)}$ is a polynomial of degree at most $j$  such that
\be\label{polbound}
|Q_j^{(n)}(i)|\le (Ci)^j\quad\mbox{for all  }i\in\N,
\ee
with $C>0$ a constant not depending on $n$ or $j$.
Assume further that  for all $J\ge 3$
\be\label{ln0}
\frac{L_n^J}{n^{J-1}\var^{\frac{J}{2}}S_n}\stackrel{n\to\infty}{\longrightarrow}0.
\ee
Then
\be\label{gencon}
U_n=\frac{S_n-\E\:S_n}{\sqrt{\var S_n}}\stackrel{d}{\to}{\cal N}(0,1),\quad\mbox{as}\quad  n\to\infty.
\ee
\end{thm}
\begin{proof} Due to Proposition \ref{gen0} we only need  to show that \eqref{fJc} holds.
The representation \eqref{ln} implies
$$f_{J,n}=\sum_{\pi\subset J}\:D_{\pi}\prod_{i=1}^J\:c_{i,n}^{m_i}=L_n^J\sum_{\pi\subset J}\:D_{\pi}e^{z_{\pi}(J,n)},
$$
where
$$
z_{\pi}(J,n)=\sum_{j\ge 1}\,\frac{A_{\pi}^{(n)}(j)}{jn^j}\quad \mbox{with}\quad A_{\pi}^{(n)}(j)=\sum_{i=1}^J\,m_i\,Q_{j+1}^{(n)}(i).
$$
Fix arbitrary $J\ge 3$. To prove \eqref{fJc}, in view of \eqref{ln0} it suffices to show that $\sum_{\pi\subset J}\:D_{\pi}e^{z_{\pi}(J,n)}$ is of order $n^{-(J-1)}$.
To do that, we expand $e^{z_\pi(J,n)}$  into power series to obtain
$$
\sum_{\pi\subset J}\:D_{\pi}e^{z_{\pi}(J,n)}=\sum_{\pi\subset J}\:D_{\pi}e^{\sum_{j\ge 1}\:\frac{1}{jn^j}A^{(n)}_{\pi}(j)}=
\sum_{\pi\subset J}\:D_{\pi}\sum_{s=0}^{\infty}\:\frac{1}{s!}\left(\sum_{j\ge 1}\:\frac{1}{jn^j}A^{(n)}_{\pi}(j)\right)^s
$$$$=\sum_{s\ge 1}\:\frac{1}{s!}\sum_{l\ge s}\:\frac{1}{n^l}\sum_{\substack{j_1,\ldots,j_s\ge 1 \\ \sum_{k=1}^s\:j_k=l}}
\frac{1}{\prod_{k=1}^s\:j_k}\sum_{\pi\subset J}\:D_{\pi}\prod_{k=1}^s\:A^{(n)}_{\pi}(j_k)\;.
$$
We claim that whenever $\sum_{k=1}^s\:j_k\le J-2$ then
\be\label{zero}
\sum_{\pi\subset J}\:D_{\pi}\prod_{k=1}^s\:A^{(n)}_{\pi}(j_k)=0.
\ee
To see this, note that  by changing the order of summation in the expression for $A_{\pi}^{(n)}(j)$ we can write it as
$$
A_{\pi}^{(n)}(j)=\sum_{k=0}^{j+1}\,\alpha_{k,j+1}^{(n)}\,H_{\pi}(k),
$$
where $(\alpha_{k,j}^{(n)})$ are the coefficients of the polynomial
$Q_j^{(n)}$, that is, 
\[Q_j^{(n)}(x)=\sum_{k=0}^j\,\alpha_{k,j}^{(n)}\,x^k.\]
Consequently, \eqref{zero} follows from identity \eqref{for7}.

To handle the terms for which $\sum_{k=1}^s\:j_k> J-2$ note that
$$|A_{\pi}^{(n)}(j)|\le (CJ)^{j+1}\sum_{i=1}^J\,m_i<K\,(CJ)^j,
$$
where $K>0$ is a constant depending only on $J$ (and not on the partition $\pi$).
Hence,
\be\label{jeden}
\left|\sum_{\pi\subset J}\:D_{\pi}\prod_{k=1}^s\:A^{(n)}_{\pi}(j_k)\right|\le \sum_{\pi\subset J}\:\left|D_{\pi}\right|\prod_{k=1}^s\:K\,(CJ)^{j_k}
\le \widetilde{C} K^s (CJ)^{\sum_{k=1}^s\:j_k}\;,
\ee
where $\widetilde{C}=\sum_{\pi\subset J}\,|D_{\pi}|$ is a constant depending only on $J$.
Therefore, restricting the sum according to \eqref{zero} and using \eqref{jeden}, we get
$$
\left|\sum_{\pi\subset J}\:D_{\pi}e^{z_{\pi}(J,n)}\right|\le
\sum_{s\ge 1}\:\frac{1}{s!}\sum_{l\ge \max\{s,J-1\}}\frac{1}{n^l}\sum_{\substack{j_1,\ldots,j_s\ge 1 \\ \sum_{k=1}^s\:j_k=l}}\:\frac{1}{\prod_{k=1}^s\:j_k}\left|\sum_{\pi\subset J}\:D_{\pi}\prod_{k=1}^s\:A^{(n)}_{\pi}(j_k)\right|
$$$$\le \widetilde{C}\sum_{s\ge 1}\:\frac{K^s}{s!}\sum_{l\ge J-1}\:\frac{1}{n^l}\,l^s(CJ)^l\;.
$$
Here we used  the inequality
$$
\sum_{\substack{j_1,\ldots,j_s\ge 1 \\\sum_{k=1}^s\:j_k=l}}\:\frac{1}{\prod_{k=1}^s\:j_k}<l^s\;,
$$
which may be seen by trivially bounding the sum by the number of its
terms. Now we change the order of summations,  arriving at
$$
\left|\sum_{\pi\subset J}\:D_{\pi}e^{z_{\pi}(J,n)}\right|\le
\widetilde{C}\sum_{l\ge J-1}\:\left(\frac{CJ}{n}\right)^l\sum_{s\ge 1}\:\frac{(lK)^s}{s!}
\le \widetilde{C}\sum_{l\ge J-1}\:\left(\frac{CJe^K}{n}\right)^l=\widetilde{C}\left(\frac{CJe^K}{n}\right)^{J-1}\sum_{l\ge 0}\:\left(\frac{CJe^K}{n}\right)^l\;.
$$
The result follows,  since for $n$ sufficiently large (such that $CJe^K<n$)  the series in the last expression converges.
\end{proof}

\begin{rem}\label{orden} A typical way Theorem~\ref{gen} will be  applied is as follows.
Assume that $\E\,S_n$ and $\var\,S_n$ are  of the same order $n$. Then obviously, \eqref{war1} and \eqref{war2} are satisfied. Assume also that \eqref{ln} and \eqref{polbound} hold and that $L_n$ is also of  order $n$.  Then clearly \eqref{ln0} is satisfied and thus \eqref{gencon} holds true.
\end{rem}

\section{Applications}\label{sec:appl}
In this section we show how the tools developed in previous section, and in particular the decomposition \eqref{ln} together with the condition \eqref{ln0}, can be conveniently used  for
proving CLTs in several situations, mostly in
 summation schemes of $\{0,\,1\}$-valued random variables, as  was indicated in Section~\ref{sec:fac_cum}.
First, four more or less standard limit results for the binomial,
negative binomial, hypergeometric and negative hypergeometric schemes
will be re--proved. Then more involved schemes of allocation problems
for distinguishable balls, indistinguishable balls, coloured balls, and random forests will be considered. The CLTs for the number of boxes with exactly $r$ balls in the first two problems and for the number of trees with exactly $r$ non-root vertices in the third problem will be derived. While the CLT in the case of distinguishable balls has been  known in the literature for years (see, e.g., \cite{ksc}),  the main focus in the other two cases appears to be on the local limit theorems (see, e.g., \cite{k1,k2,p_bk}). We have not found any references for the asymptotic normality results for the problems we consider in GIAS models.

The models in Sections~\ref{indballs}--\ref{forest} are examples of the  {\em generalized allocation schemes} (GAS), that is,
\be\lbl{eq:gas}
\P(\xi_1^{(n)}=k_1,\ldots,\xi_N^{(n)}=k_N)=\P(\eta_1=k_1,\ldots,\eta_N=k_N|\eta_1+\ldots+\eta_N=n),
\ee
where $\eta_1,\ldots,\eta_N$ are independent random variables.

On the other hand the models in Sections~\ref{neg_mult} and \ref{dir} are examples  of what may be called the {\em generalized inverse allocation schemes} (GIAS), that is,
\be\lbl{eq:gias}
\P(\xi_1^{(n)}=k_1,\ldots,\xi_N^{(n)}=k_N)= C\, \P(\eta_1=k_1,\ldots,\eta_N=k_N|\eta_0+\eta_1+\ldots+\eta_N=n+k_1+\ldots+k_N),
\ee
where $\eta_0,\eta_1,\ldots,\eta_N$ are independent random variables,
$C$ is a proportionality constant and the equality is understood to
hold whenever the right hand--side is summable. This last requirement
is not vacuous: if, e.g., $N=1$ and $\eta_0=n$ a.s.,  then trivially
the probability on the right--hand side of \eqref{eq:gias} is 1
regardless of the value of $k_1$, and hence these probabilities are not summable as long as $\eta_1$ takes on infinitely many values.

In practical situations of GAS models $(\eta_j)$ are identically
distributed, and in the case of GIAS we assume that the $\eta_j$ have the same distribution for $j=1,\dots, N$, which may differ from the distribution of $\eta_0$.

In the derivations below we will often use the following expansion
\be\label{expan}
\left(1-\frac{a}{b}\right)^c=e^{c\log\left(1-\frac{a}{b}\right)}=e^{-c\sum\limits_{j=1}^{\infty}\frac{a^j}{jb^j}}\;,
\ee
which is valid for any $0<|a|<|b|$ and any real $c$. We also recall
(see,  e.g., \cite[Chapter~ 6.5]{gkp}) that
\be\label{bern}
\Q_{j+1}(M):=\sum_{k=1}^{M-1}\,k^j=\frac{1}{j+1}\sum_{l=1}^{j+1}\,\binom{j+1}{l}\,B_{j+1-l}\,M^l,
\ee
where $(B_k)$ are the Bernoulli numbers.  Clearly, $\Q_j$ is a polynomial of degree $j$ satisfying \eqref{polbound} with $C=1$.
For notational convenience we let
$$
T(m,t)=\prod_{k=1}^{m-1}\,\left(1-\frac{k}{t}\right)
$$
for $t>0$ and integer $m>0$. It follows from \eqref{expan} and \eqref{bern}  that for $t>m$
\be\label{tmz}
T(m,t)=e^{-\sum_{j\ge 1}\,\frac{1}{jt^j}\Q_{j+1}(m)}.
\ee

\subsection{Classical discrete distributions}
In this subsection we re-derive asymptotic normality of
$$\frac{S_n-\E\,S_n}{\sqrt{\var\,S_n}}$$ for laws of $S_n$ belonging to four classical families of discrete distributions: binomial, negative binomial, hypergeometric and negative hypergeometric.
\subsubsection{Binomial scheme}
Let $(\eps_i)$ be a sequence of
iid Bernoulli random variables, $P(\eps_1=1)=p=1-P(\eps_1=0)$. Then $S_n=\sum_{i=1}^n\:\eps_i$ has the binomial $b(n,p)$ distribution.
To see how Theorem \ref{gen} allows us  to re-derive de Moivre-Laplace theorem
\be\label{ML}
\frac{S_n-np}{\sqrt{np(1-p)}}\stackrel{d}{\to}{\cal N}(0,1),
\ee
in a simple way, we first set $L_n=np$. Then $\E\,S_n=L_n$ and $\var\,S_n=L_n(1-p)$. Furthermore, $\P(\eps_1=\ldots=\eps_i=1)=p^i$ and thus by
\eqref{ce_simp} it follows that the $i$th factorial moment of $S_n$ is
\be\label{cml} c_{i,n}=\E(S_n)_i=(n)_ip^i=L_n^i\,T(i,n)\;.\ee
Thus \eqref{tmz} implies representation \eqref{ln}  with
$Q_{j+1}=-\Q_{j+1}$ and \eqref{ML} follows from Remark \ref{orden}.

\subsubsection{Negative binomial scheme}
Let  $S_n$ denote the number of failures until the $n$th success in
Bernoulli trials,  with $p$ being the probability of a success in a single trial, that is, $S_n$ is negative binomial $nb(n,p)$ with
$$
\P(S_n=k)=\binom{n+k-1}{k}(1-p)^kp^n\quad k=0,1,\ldots.
$$

We will show how Theorem \ref{gen} allows us 
to re-derive the CLT  for $(S_n)$ in a simple way, which states that for $n\to\infty$
\be\label{ctgnb}
\frac{pS_n-n(1-p)}{\sqrt{n(1-p)}}\stackrel{d}{\to}{\cal N}(0,1)\;.
\ee
Set $L_n=n\frac{1-p}{p}$ so that $\E\,S_n=L_n$ and $\var\,S_n=\frac{L_n}{p}$. Furthermore,  the  $i$th factorial moment of $S_n$ is easily derived as
$$
c_{i,n}=\E(S_n)_i=L_n^i\,T(i,-n)\;.
$$
Hence \eqref{ln} holds with
$Q_{j+1}=(-1)^{j+1}\Q_{j+1}$ and thus \eqref{ctgnb} follows from Remark \ref{orden}.

\subsubsection{Hypergeometric scheme}
From an urn containing $N$ white and $M$ black balls we draw subsequently without replacement $n$ balls ($n\le \min\{M,N\}$). For $i=1,\ldots,n$, let $\eps_i=1$ if a white ball is drawn at the $i$th drawing and let  $\eps_i=0$ otherwise. Then $S_n=\sum_{i=1}^n\,\eps_i$ has a hypergeometric distribution $Hg(N,M;n)$, that is,
$$
\P(S_n=k)=\frac{\binom{N}{k}\,\binom{M}{n-k}}{\binom{N+M}{n}}\;,\quad k=0,1,\ldots,n\;.
$$

Using again Theorem \ref{gen} we will show that under the assumptions 
$N=N(n)\to\infty$, $M=M(n)\to\infty$,  and $\frac{N}{N+M}\to p\in(0,1)$ with $n\to\infty$
\be\label{hypas}
\frac{(N+M)S_n-nN}{\sqrt{nNM(N+M-n)/(N+M-1)}}\stackrel{d}{\to} \mathcal{N}(0,1).
\ee
Setting $L_n=\frac{nN}{N+M}$ we have
\be\label{ev}
\E\,S_n=L_n\quad\mbox{and}\quad \var\,S_n=L_n\frac{M(N+M-n)}{(N+M)(N+M-1)}\;.
\ee
Moreover, on noting that $(\eps_1,\ldots,\eps_n)$ is exchangeable by \eqref{ce_simp} we get
$$
c_{i,n}=\E\,(S_n)_i=(n)_i\P(\eps_1=\ldots=\eps_i=1)=\frac{(n)_i(N)_i}{(N+M)_i}=L_n^i\,
\frac{T(i,n)\,T(i,N)}{T(i,N+M)}.
$$
As in earlier schemes we obtain representation \eqref{ln} with
$$
Q_{j+1}=\left[-1-\left(\frac{n}{N}\right)^j+\left(\frac{n}{N+M}\right)^j\right]\Q_{j+1}.
$$
Moreover, the condition \eqref{ln0} is fulfilled since $\E\,S_n$, $\var\,S_n$ and $L_n$ are all of order $n$. See again Remark~\ref{orden} to conclude that \eqref{hypas} holds true.

\subsubsection{Negative hypergeometric scheme}
Let  $S_n$ be a random variable with negative hypergeometric distribution of the first kind, that is,
$$
\P(S_n=k)=\binom{n}{k}\frac{B(\alpha_n+k,\beta_n+n-k)}{B(\alpha_n,\beta_n)}\quad k=0,1,\ldots,n,
$$
with $\alpha_n=n\alpha$ and $\beta_n=n\beta$.
The CLT for $S_n$  states that for $n\to\infty$
\be\label{ctgnhg1}
\frac{(\alpha+\beta)^{\frac{3}{2}}S_n-n\alpha\sqrt{\alpha+\beta}}{\sqrt{n\alpha\beta(1+\alpha+\beta)}}\stackrel{d}{\to}{\cal N}(0,1)\;.
\ee
To quickly derive it from Theorem \ref{gen}, let 
$L_n=\frac{n\alpha}{\alpha+\beta}$ and note that $$\E\,S_n=L_n\qquad\mbox{and}\qquad\var\,S_n=L_n\frac{n\beta(1+\alpha+\beta)}{(\alpha+\beta)^2(n\alpha+n\beta+1)}.$$
Further,  the $i$th factorial moment of $S_n$ is easily derived as
$$ c_{i,n}=\E(S_n)_i=L_n^i\,\frac{T(i,n)\,T(i,-\alpha n)}{T(i,-(\alpha+\beta)n)}\;.
$$
Thus,  again due to \eqref{tmz} we conclude that  representation \eqref{ln} holds with
$$
Q_{j+1}=\left(-1-\frac{(-1)^j}{\alpha^j}+\frac{(-1)^j}{(\alpha+\beta)^j}\right)
\Q_{j+1}(i).
$$
The final result follows by Remark~\ref{orden}.

\subsection{Asymptotics of occupancy in generalized allocation schemes (GAS)}\lbl{sec:gas}
In this subsection we will derive asymptotics for 
\[S_n^{(r)}=\sum_{i=1}^N\,I(\xi_i^{(n)}=r)\] in several generalized allocation schemes as defined at the beginning of Section \ref{sec:appl}. As we will see,  when $n\to \infty$ and  $N/n\to\lambda\in(0,\infty]$ the order of $\E\,S_n^{(r)}$ is $n^r/N^{r-1}$  for any $r=0,1,\ldots$, and the order of  $\var\,S_n^{(r)}$ is the same for $r\ge2$. When $\lambda=\infty$ and $r=0$ or  $1$ the order of $\var\,S_n^{(r)}$   is $n^2/N$. Consequently, we will derive asymptotic normality of
$$
\frac{S_n-\E\,S_n^{(r)}}{\sqrt{n^r/N^{r-1}}}
$$
when either
\begin{itemize}
\item[{\em (a)}] $r\ge 0$ and $\lambda<\infty$  or,
\item[{\em (b)}] $r\ge 2$, $\lambda=\infty$ and  $\frac{n^r}{N^{r-1}}\to\infty$,
\end{itemize}
and  asymptotic normality of
$$
\sqrt{N}\,\frac{S_n^{(r)}-\E\,S_n^{(r)}}{n}
$$
when $\lambda=\infty$, $\frac{n^2}{N}\to\infty$ and $r=0,1$.

Although in all the cases results look literally the same (with
different asymptotic expectations and variances and having different
proofs) for the sake of precision we decided to repeat formulations of
theorems in each of the subsequent cases.
\subsubsection{Indistinguishable balls}\label{indballs}
Consider a scheme of a random distribution of $n$
indistinguishable balls into $N$ distinguishable boxes, such that
all distributions are equiprobable. 
That is, if $\xi_i=\xi_i^{(n)}$ denotes
the number of balls which fall into the $i$th box, $i=1,\ldots,N$,
then
$$
\P(\xi_1=i_1,\ldots,\xi_N=i_N)=\binom{n+N-1}{n}^{-1}
$$
for any $i_k\ge 0$, $k=1,\ldots,N$, such that $i_1+\ldots+i_N=n$. Note that this is GAS and that \eqref{eq:gas} holds with $\eta_i\sim Geom(p)$, $0<p<1$.

Let 
\[S_n^{(r)}=\sum_{i=1}^N\:I(\xi_i=r)\] 
denote the number of boxes with exactly $r$ balls.
 Note that the distribution of $(\xi_1,\ldots,\xi_N)$ is
exchangeable. Moreover,
$$
\P(\xi_1=\ldots=\xi_i=r)=\frac{\binom{n-ri+N-i-1}{n-ri}}{\binom{n+N-1}{n}}\;.
$$
Therefore by \eqref{ce_simp} we get  \be\label{ce}
c_{i,n}=\E(S_n^{(r)})_i=\frac{(N)_i(N-1)_i(n)_{ir}}{(n+N-1)_{i(r+1)}}. \ee
 Consequently,
$$
\E\,S_n^{(r)}=c_{1,n}=\frac{N(N-1)(n)_r}{(n+N-1)_{r+1}},
$$
and since $\var\,S_n^{(r)}=c_{2,n}-c_{1,n}^2+c_{1,n}$ we have
$$
\var\,S_n^{(r)}=\frac{N(N-1)^2(N-2)(n)_{2r}}{(n+N-1)_{2r+2}}
-\left(\frac{N(N-1)(n)_r}{(n+N-1)_{r+1}}\right)^2+\frac{N(N-1)(n)_r}{(n+N-1)_{r+1}}.
$$

In the asymptotics below we consider the situation when $n\to \infty$ and  $\frac{N}{n}\to\lambda\in(0,\infty].$
Then for any integer $r\ge 0$
\be\label{ase}
\frac{N^{r-1}}{n^r}\E\,S_n^{(r)}\to \left(\frac{\lambda}{1+\lambda}\right)^{r+1}\quad(=1\;\;\mbox{for}\;\;\lambda=\infty )\;.
\ee
It is also elementary but more laborious to prove that,  for any $r\ge 2$ and $\lambda\in(0,\infty]$ or $r=0,1$ and $\lambda\in(0,\infty)$
\be\label{varr}
\frac{N^{r-1}}{n^r}\var\,S_n^{(r)}\to
\left(\frac{\lambda}{1+\lambda}\right)^{r+1}\left(1-\frac{\lambda\left(1+\lambda+(\lambda
r-1)^2\right)}{(1+\lambda)^{r+2}}\right)=:\sigma^2_r \quad (=1\;\;\mbox{for}\;\;\lambda=\infty)\;.
\ee
Further, for $\lambda=\infty$
$$
\frac{N}{n^2}\var\,S_n^{(0)}\to 1=: \tilde{\sigma}_0^2\quad\mbox{and}\quad \frac{N}{n^2}\var\,S_n^{(1)}\to 4=: \tilde{\sigma}_1^2.
$$
Similarly, in this case
$$
\frac{N}{n^2}\,\cov(S_n^{(0)},S_n^{(1)})\to -2\quad \mbox{and}\quad \frac{N}{n^2}\,\cov(S_n^{(0)},S_n^{(2)})\to 1,
$$
and thus for the correlation coefficient we have
\be\label{asco}
\rho(S_n^{(0)},S_n^{(1)})\to -1 \quad \mbox{and}\quad \rho(S_n^{(0)},S_n^{(2)})\to 1.
\ee

Now we are ready to deal with CLTs.
\begin{thm}\lbl{dibal}
Let $N/n\to\lambda\in(0,\infty]$. Let either \newline {\em (a)} $r\ge 0$ and $\lambda<\infty$,  or \newline {\em (b)} $r\ge 2$, $\lambda=\infty$ and  $\frac{n^r}{N^{r-1}}\to\infty$.

 Then
$$
\frac{S_n^{(r)}-\E\,S_n^{(r)}}{\sqrt{n^r/N^{r-1}}}\stackrel{d}{\to}{\cal
N}(0,\sigma^2_r)\;.
$$
\end{thm}

\begin{proof}
Note that \eqref{ce} can be rewritten as
$$
c_{i,n}=L_n^i\,\frac{T(i,N)\,T(i,N-1)\,T(ir,n)}
{T(i(r+1),n+N-1)}\quad\mbox{with}\quad
L_n=\frac{N(N-1)n^r}{(n+N-1)^{r+1}}.
$$
Therefore,  as in the previous cases, using \eqref{tmz} we conclude that representation \eqref{ln} holds with
$$
Q_{j+1}(i)=-\left[\left(\frac{n}{N}\right)^j+\left(\frac{n}{N-1}\right)^j\right]\Q_{j+1}(i)-\Q_{j+1}(ri)+\left(\frac{n}{n+N-1}\right)^j\Q_{j+1}((r+1)i).
$$
To conclude the proof we note that $\E\,S_n^{(r)}$, $\var\,S_n^{(r)}$
and $L_n$ are of the same order, $n^r/N^{r-1}$,  and use Remark~\ref{order} stated below.
\end{proof}

\begin{rem}\label{order} If $\E\,S_n^{(r)}$ and $\var\,S_n^{(r)}$ are of the same order and diverge to $\infty$, then \eqref{war1} and \eqref{war2} hold.
Moreover, if $L_n$ and $\var\,S_n^{(r)}$ are both of order $n^r/N^{r-1}$ then the left--hand side of \eqref{ln0} is of order
$$
\frac{1}{n^{J-1}}\,\left(\frac{n^r}{N^{r-1}}\right)^{\frac{J}{2}}=\left(\frac{n}{N}\right)^{\frac{J}{2}(r-1)}\,\frac{1}{n^{\frac{J}{2}-1}}.
$$
That is, when  either  $\lambda\in(0,\infty)$ and $r=0,1,\ldots$ or  $\lambda=\infty$ and $r= 2,3 \ldots$ the condition \eqref{ln0} is satisfied.
\end{rem}

In the remaining cases we use asymptotic correlations

\begin{thm}
Let $N/n\to\infty$ and $n^2/N\to\infty$.
Then for $r=0,1$
$$
\sqrt{N}\,\frac{S_n^{(r)}-\E\,S_n^{(r)}}{n}\stackrel{d}{\to}{\cal
N}\left(0,\tilde{\sigma_r}^2\right)\;.
$$
\end{thm}

\begin{proof}
Due to the second equation in \eqref{asco}, it follows that
$$
\sqrt{N}\,\frac{S_n^{(0)}-\E\,S_n^{(0)}}{\tilde{\sigma}_0 n}-\sqrt{N}\,\frac{S_n^{(2)}-\E\,S_n^{(2)}}{\sigma_2 n}\stackrel{L^2}{\to}0.
$$
Therefore the result for $r=0$ holds.
Similarly, for $r=1$ it suffices to observe that the first equation in \eqref{asco} implies
$$
\sqrt{N}\,\frac{S_n^{(0)}-\E\,S_n^{(0)}}{\tilde{\sigma}_0 n}+\sqrt{N}\,\frac{S_n^{(1)}-\E\,S_n^{(1)}}{\tilde{\sigma}_1 n}\stackrel{L^2}{\to}0.
$$
\end{proof}

\subsubsection{Distinguishable balls}\label{dibal}
Consider a scheme of a random distribution of $n$ distinguishable
balls into $N$ distinguishable boxes, such that any such
distribution is equally likely. Then, if $\xi_i=\xi_i^{(n)}$ denotes the number
of balls which fall into the $i$th box, $i=1,\ldots,N$,
$$
\P(\xi_1=i_1,\ldots,\xi_N=i_N)=\frac{n!}{i_1!\ldots i_N!}N^{-n}
$$
for any $i_l\ge 0$, $l=1,\ldots,N$, such that $i_1+\ldots+i_N=n$. This is a GAS with $\eta_i\sim Poisson(\lambda)$, $\lambda>0$, in \eqref{eq:gas}.

For a fixed non--negative integer $r$ let
$$
S_n^{(r)}=\sum_{i=1}^N\:I(\xi_i=r)
$$
be the number of boxes with exactly $r$ balls.
 Obviously, the distribution of $(\xi_1,\ldots,\xi_N)$ is
exchangeable, and
$$
\P(\xi_1=\ldots=\xi_i=r)=\frac{n!}{(r!)^i(n-ir)!}N^{-ir}\left(1-\frac{i}{N}\right)^{n-ir}\;.
$$
Therefore by \eqref{ce_simp} we get  \be\label{cee}
c_{i,n}=\E\,(S_n^{(r)})_i=\frac{(N)_i(n)_{ir}}{(r!)^i
N^{ri}}\left(1-\frac{i}{N}\right)^{n-ir}.\ee
Consequently, for any $r=0,1,\ldots$
$$
\E\,S_n^{(r)}=c_{1,n}=\frac{(n)_r\left(1-\frac{1}{N}\right)^{n-r}}{r!N^{r-1}}
$$
and
$$
\var\,S_n^{(r)}=c_{2,n}-c_{1,n}^2+c_{1,n}=\frac{(N-1)(n)_{2r}\left(1-\frac{2}{N}\right)^{n-2r}}{(r!)^2
N^{2r-1}}-\frac{(n)_r^2\left(1-\frac{1}{N}\right)^{2(n-r)}}{(r!)^2N^{2(r-1)}}+\frac{(n)_r\left(1-\frac{1}{N}\right)^{n-r}}{r!N^{r-1}}.
$$
In the asymptotics below we consider the situation when $n\to \infty$ and  $\frac{N}{n}\to\lambda\in(0,\infty].$
Then,  for any integer $r\ge 0$,
\be\lbl{mrr}
\lim_{n\to\infty}\,\frac{N^{r-1}}{n^r}\E\,S_n^{(r)}=\frac{1}{r!}e^{-\frac{1}{\lambda}}\quad \left(=\frac{1}{r!}\;\;\mbox{for}\;\;\lambda=\infty\right).
\ee
It is also elementary but more laborious to check that, for any fixed $r\ge 2$ and $\lambda\in(0,\infty]$ or $r=0,1$ and $\lambda\in(0,\infty)$,  \be\lbl{varr_super}
\lim_{n\to \infty}
\frac{N^{r-1}}{n^r}\var\,S_n^{(r)}=\frac{e^{-\frac{1}{\lambda}}}{r!}\left(1-\frac{e^{-\frac{1}{\lambda}}(\lambda+(\lambda r-1)^2)}{r!\lambda^{r+1}}\right):=\sigma_r^2\quad \left(=\frac{1}{r!}\;\;\mbox{for}\;\;\lambda=\infty\right). \ee
Further, for $\lambda=\infty$, 
$$
\frac{N}{n^2}\,\var\,S_n^{(0)}\to \frac{1}{2}=:\tilde{\sigma}_0^2\quad\mbox{and}\quad \frac{N}{n^2}\,\var\,S_n^{(1)}\to 2=:\tilde{\sigma}_1^2.
$$
Similarly one can prove that
$$
\frac{N}{n^2}\,\cov(S_n^{(0)},S_n^{(1)})\to -1.
$$
Therefore, for the correlation coefficients we have
\be\label{ro1}
\rho(S_n^{(0)},S_n^{(1)})\to -1.
\ee
Since
$$
\frac{N}{n^2}\,\var\,S_n^{(2)}\to \frac{1}{2}=\sigma_2^{(2)}\quad\mbox{and}\quad \frac{N}{n^2}\,\cov(S_n^{(0)},S_n^{(2)})\to \frac{1}{2}
$$
we also have
\be\label{ro2}
\rho(S_n^{(0)},S_n^{(2)})\to 1.
\ee
We  consider the cases when $r\ge 2$ and $\lambda=\infty$ or $r\ge0$ and $\lambda\in(0,\infty)$.
\begin{thm}\label{indibal}
Let $N/n\to\lambda\in(0,\infty]$. Let either \newline {\em (a)} $r\ge 0$ and $\lambda<\infty$,  or \newline {\em (b)} $r\ge 2$, $\lambda=\infty$ and  $\frac{n^r}{N^{r-1}}\to\infty$.

Then
$$
\frac{S_n^{(r)}-\E\,S_n^{(r)}}{\sqrt{n^r/N^{r-1}}}\stackrel{d}{\to}{\cal
N}\left(0,\sigma_r^2\right)\;.
$$
\end{thm}
\begin{proof}
Write $c_{i,n}$ as
$$
c_{i,n}=L_n^i\,e^{i\frac{n}{N}}\,\left(1-\frac{i}{N}\right)^{n-ir}\,T(i,N)\,T(ir,n),\quad\qquad\mbox{where}\quad\quad
L_n=\frac{n^re^{-\frac{n}{N}}}{r!N^{r-1}}.
$$
Then, the representation \eqref{ln} holds with
$$
Q_{j+1}(i)= \left(r-\frac{j}{j+1}\,\frac{n}{N}\right)\left(\frac{n}{N}\right)^j\,i^{j+1}-\left(\frac{n}{N}\right)^j\Q_{j+1}(i)-\Q_{j+1}(ri).
$$
Since $\E\,S_n^{(r)}$, $\var\,S_n^{(r)}$ and  $L_n$ are of order
$n^r/N^{r-1}$,  the final result follows by Remark \ref{order}.
\end{proof}

As in the case of indistinguishable balls, using \eqref{ro2} and \eqref{ro1} we get
 the following
\begin{thm}\label{thm:diballs_r01}
Let $N/n\to\infty$ and $n^2/N\to\infty$.
Then for $r=0,1$
$$
\sqrt{N}\,\frac{S_n^{(r)}-\E\,S_n^{(r)}}{n}\stackrel{d}{\to}{\cal
N}\left(0,\tilde{\sigma_r}^2\right)\;.
$$
\end{thm}

\subsubsection{Coloured balls}
An urn contains $NM$ balls, $M$ balls of each of $N$ colors. From the urn a simple random sample of $n$ elements is drawn. We want to study  the asymptotics of the number of colors with exactly $r$ balls in the sample. More precisely, let $\xi_i=\xi_i^{(n)}$ denote the number of balls of color $i$, $i=1,\ldots,N$. Then
$$
\P(\xi_1=k_1,\ldots,\xi_{N}=k_N)=\frac{\prod_{i=1}^N\,\binom{M}{k_i}}{\binom{NM}{n}}
$$
for all integers $k_i\ge 0$, $i=1,\ldots,N$, such that $\sum_{i=1}^N\,k_i=n$. Obviously, the random vector $(\xi_1,\ldots,\xi_N)$ is exchangeable and the GAS equation \eqref{eq:gas} holds with $\eta_i\sim b(M,p)$, $0<p<1$.

For an integer $r\ge 0$ we want to study the asymptotics of
$$
S_n^{(r)}=\sum_{i=1}^N\,I(\xi_i=r).
$$
For the $i$the factorial moment we get
\be\label{ccee}
c_{i,n}=(N)_i\,\P(\xi_1=\ldots\xi_i=r)=(N)_i\frac{\binom{M}{r}^i\binom{(N-i)M}{n-ri}}{\binom{NM}{n}}.
\ee
Consequently, for any $r=0,1,\ldots$
$$
\E\,S_n^{(r)}=c_{1,n}=N\frac{\binom{M}{r}\binom{(N-1)M}{n-r}}{\binom{NM}{n}}
$$
and
$$
\var\,S_n^{(r)}=c_{2,n}-c_{1,n}^2+c_{1,n}=N(N-1)\frac{\binom{M}{r}^2\binom{(N-2)M}{n-2r}}{\binom{NM}{n}}
-N^2\frac{\binom{M}{r}^2\binom{(N-1)M}{n-r}^2}{\binom{NM}{n}^2}+N\frac{\binom{M}{r}\binom{(N-1)M}{n-r}}{\binom{NM}{n}}.
$$
In the asymptotics below we consider the situation when $n\to \infty$, $\frac{N}{n}\to\lambda\in(0,\infty]$, and $M=M(n)\ge n$.

Although the computational details are different,  asymptotic formulas for $\E\,S_n^{(r)}$, $\var\,S_n^{(r)}$, $\cov(S_n^{(0)},\,S_n^{(1)})$ and $\cov(S_n^{(0)},S_n^{(2)})$ are literally the same as for their counterparts  in the case of occupancy for distinguishable balls studied in Subsection~\ref{dibal}.

First we will consider the limit result in the case $r\ge 2$, $\lambda=\infty$, and $r\ge 0$, $\lambda\in(0,\infty)$.
\begin{thm}
Let $N/n\to\lambda\in(0,\infty]$ and $M=M(n)\ge n$. Let either
\newline {\em (a)} $r\ge 0$ and $\lambda<\infty$,  or \newline {\em (b)} $r\ge 2$, $\lambda=\infty$ and  $\frac{n^r}{N^{r-1}}\to\infty$.

Then
$$
\frac{S_n^{(r)}-\E\,S_n^{(r)}}{\sqrt{n^r/N^{r-1}}}\stackrel{d}{\to}{\cal
N}\left(0,\sigma_r^2\right)\;.
$$
\end{thm}
\begin{proof}
Rewrite the formula \eqref{ccee} as
$$
c_{i,n}=L_n^i\frac{T((M-r)i,NM-n)\,T(i,N)\,T(ri,n)}{T(Mi,NM)}\quad\mbox{with}\quad
L_n=Nn^r\binom{M}{r}\frac{\left(1-\frac{n}{NM}\right)^M}{(NM-n)^r}.
$$
Thus the representation \eqref{ln} holds with
$$
Q_{j+1}(i)=\left(\frac{n}{NM}\right)^j\Q_{j+1}(Mi)-\left(\frac{n}{NM-n}\right)^j\Q_{j+1}((M-r)i)-\left(\frac{n}{N}\right)^j\Q_{j+1}(i)-\Q_{j+1}(ri).
$$
We need to  see that the polynomials $Q_j$ satisfy bound \eqref{polbound}. This is clearly true for each of the last two terms in the above expression for $Q_{j+1}$.  For the first two terms we have
\begin{eqnarray*}&&
\left|\left(\frac{n}{NM}\right)^j\Q_{j+1}(Mi)-\left(\frac{n}{NM-n}\right)^j\Q_{j+1}((M-r)i)\right|\\&&\quad\quad
=\left|\left(\frac{n}{NM}\right)^j\sum_{k=1}^{Mi-1}\,k^j-\left(\frac{n}{NM-n}\right)^j\sum_{k=1}^{(M-r)i-1}\,k^j\right|\\&&\quad\quad\le
\left|\left(\frac{n}{NM}\right)^j-\left(\frac{n}{NM-n}\right)^j\right|\Q_{j+1}(Mi)+\left(\frac{n}{NM-n}\right)^j\sum_{k=(M-r)i}^{Mi-1}\,k^j.
\end{eqnarray*}
Since
$$
\left|\left(\frac{n}{NM}\right)^j-\left(\frac{n}{NM-n}\right)^j\right|\le \left(\frac{n}{NM-n}\right)^j\,\frac{jn}{NM}\le\left(\frac{2n}{NM-n}\right)^j\,\frac{n}{NM},
$$
and
$$
\sum_{k=(M-r)i}^{Mi-1}\,k^j\le rM^ji^{j+1},
$$
and $\Q_{j+1}(Mi)\le M^{j+1}\,i^{j+1}$, 
we conclude that the $Q_j$ do satisfy \eqref{polbound}.

Clearly, $\E\,S_n^{(r)}$, $\var\,S_n^{(r)}$ and $L_n$ are of order $n^r/N^{r-1}$, and again we conclude the proof by referring to Remark \ref{order}.
\end{proof}

Asymptotic normality for $S_n^{(1)}$ and $S_n^{(0)}$ for
$\lambda=\infty$ also holds with  an
 identical statement  to that of Theorem~\ref{thm:diballs_r01} for distinguishable balls.

\subsubsection{Rooted trees in random forests}\label{forest}
Let ${\cal T}(N,n)$ denote a forest with $N$ roots (that is,  $N$
rooted trees) and $n$ non-root vertices. Consider a uniform
distribution on the set of such ${\cal T}(N,n)$ forests. Let
$\xi_i=\xi_i^{(n)}$, $i=1,\ldots,N$, denote the number of non-root
vertices in the $i$th tree. Then (see,  e.g.,  \cite{cf}, \cite{p} or
\cite{p_bk}),  for any $k_i\ge 0$ such that  $\sum_{i=1}^N\,k_i= n$
$$
\P(\xi_1=k_1,\ldots,\xi_{N}=k_{N})=\frac{n!}{\prod_{i=1}^N\,k_i}\,
\frac{\prod_{i=1}^N\,(k_i+1)^{k_i-1}}{N(N+n)^{n-1}}.
$$
Note that this distribution is exchangeable and that it is a GAS with $\eta_i$ in \eqref{eq:gas} given by
\[\P(\eta_i=k)=\frac{\lambda^k(k+1)^{k-1}}{k!}e^{-(k+1)\lambda},\quad k=0,1\dots;\quad \lambda>0.\]
We mention in passing that the distribution of $\eta_i$  may be
identified as an Abel distribution discussed in \cite{lm} with (in their notation)  $p=1$ and $\theta=\ln\lambda-\lambda$. We refer to \cite[Example D]{lm}  for more information on Abel distributions, including further references.

For a fixed number $r\ge 0$ we are interested in  the number $S_n^{(r)}$ of trees with exactly $r$ non-root vertices:
$$
S_n^{(r)}=\sum_{i=1}^N\,I(\xi_i=r).
$$
Since the $i$th factorial moment of $S_n^{(r)}$ is of the form
$$
c_{i,n}=(N)_i\,\P(\xi_1=\ldots=\xi_i=r),
$$
we have to find the marginal distributions of the random vector $(\xi_1,\ldots,\xi_{N})$. From the identity
$$
s\sum_{k=0}^m\,\binom{m}{k}\,(k+1)^{k-1}(m-k+s)^{m-k-1}=(s+1)(m+1+s)^{m-1},
$$
which is valid for any natural $m$ and $s$, we easily obtain that, for $k_j\ge 0$ such that $\sum_{j=1}^{i+1}\,k_j=n$,
$$
\P(\xi_1=k_1,\ldots,\xi_i=k_i)=\frac{n!}{\prod_{j=1}^{i+1}\;k_j!}\,\frac{(N-i)(k_{i+1}+N-i)^{k_{i+1}-1}\,\prod_{j=1}^i\,(k_j+1)^{k_j-1}}{N(N+n)^{n-1}}.
$$
Therefore
\be\label{cec}
c_{i,n}=\frac{(r+1)^{i(r-1)}}{(r!)^i}\,\frac{N-i}{N}\,\frac{(N)_i(n)_{ri}}{(n+N-(r+1)i)^{ri}}\,\left(1-\frac{(r+1)i}{n+N}\right)^{n-1}.
\ee
Hence
$$
\E\,S_n^{(r)}=c_{1,n}=\frac{(r+1)^{r-1}}{r!}\,\frac{(N-1)(n)_r}{(n+N-r-1)^r}\,\left(1-\frac{r+1}{n+N}\right)^{n-1}.
$$
Thus, if $N/n\to\lambda\in(0,\infty]$ we have
$$
\frac{N^{r-1}}{n^r}\,\E\,S_n^{(r)}\to\frac{(r+1)^{r-1}}{r!}\,\left(\frac{\lambda}{\lambda+1}\right)^r\,e^{-\frac{r+1}{\lambda+1}}\;\;
\left(=\frac{(r+1)^{r-1}}{r!}\;\mbox{for}\;\lambda=\infty\right).
$$
Since $\var\,S_n^{(r)}=c_{2,n}-c_{1,n}^2+c_{1,n}$,  elementary but cumbersome computations lead to
$$
\frac{N^{r-1}}{n^r}\,\var\,S_n^{(r)}\to \sigma_r^2
$$$$=\frac{e^{-\frac{r+1}{\lambda+1}}}{(r+1)!}\,\left(\frac{\lambda (r+1)}{\lambda+1}\right)^r
\left[1-\frac{e^{-\frac{r+1}{\lambda+1}}}{(r+1)!}\,\left(\frac{r+1}{\lambda+1}\right)^r\right]-\lambda\left(\frac{e^{-\frac{r+1}{\lambda+1}}(\lambda r-1)}{(r+1)!(\lambda+1)}\right)^2\left(\frac{\lambda (r+1)^2}{(\lambda+1)^2}\right)^r
$$
for $r\ge 2$ and $\lambda\in(0,\infty]$ and for $r=0,1$ and $\lambda\in(0,\infty)$.
For $r=0,1$ and $\lambda=\infty$, and $n^2/N\to\infty$ we have
$$
\frac{N}{n^2}\,\var\,S_n^{(0)}\to \frac{3}{2}=\tilde{\sigma}_0^2\quad\mbox{and}\quad \frac{N}{n^2}\,\var\,S_n^{(1)}\to 6=\tilde{\sigma}_1^2.
$$
Similarly one can prove that
$$
\frac{N}{n^2}\,\cov(S_n^{(0)},S_n^{(1)})\to -3.
$$
Therefore, for the correlation coefficients we have
\be\label{ro01}
\rho(S_n^{(0)},S_n^{(1)})\to -1.
\ee
Since
$$
\frac{N}{n^2} \var\,S_n^{(2)}\to \frac{3}{2}=\sigma_2^{(2)}\quad\mbox{and}\quad \frac{N}{n^2}\,\cov(S_n^{(0)},S_n^{(2)})\to 1,
$$
we also have
\be\label{ro02}
\rho(S_n^{(0)},S_n^{(2)})\to 1.
\ee
\begin{thm}
Let $N/n\to\lambda\in(0,\infty]$. Let either \newline {\em (a)} $r\ge 0$ and $\lambda<\infty$, or \newline {\em (b)} $r\ge 2$, $\lambda=\infty$ and  $\frac{n^r}{N^{r-1}}\to\infty$.

Then
$$
\frac{S_n^{(r)}-\E\,S_n^{(r)}}{\sqrt{n^r/N^{r-1}}}\stackrel{d}{\to}{\cal
N}\left(0,\sigma_r^2\right)\;.
$$
\end{thm}
\begin{proof}
Since the asymptotics of $\var S_n^{(r)}$ and $\E\,S_n^{(r)}$ is of the same order as in Theorem \ref{dibal}, the conditions \eqref{war1} and \eqref{war2} are satisfied.
Using   \eqref{cec} we write
$$
c_{i,n}=L_n^i\,e^{i\frac{(n-1)(r+1)}{n+N}}\,\left(1-\frac{(r+1)i}{n+N}\right)^{n-1-ri}
\,T(i+1,N)\,T(ri,n),
$$
where
$$
L_n=\frac{N(r+1)^{r-1}}{r!}\left(\frac{n}{n+N}\right)^re^{-\frac{(n-1)(r+1)}{n+N}}.
$$
Thus the representation \eqref{ln} holds true with
$$
Q_{j+1}(i)=\left(r-\frac{j(r+1)(n-1)}{(j+1)(n+N)}\right)\,\left(\frac{(r+1)(n-1)}{n+N}\right)^j\,i^{j+1}-\left(\frac{n}{N}\right)^j\Q_{j+1}(i+1)-\Q_{j+1}(ri).
$$
The final result follows again by Remark \ref{order}, on noting that $\E\,S_n^{(r)}$, $\var\,S_n^{(r)}$ and $L_n$ are of order $n^r/N^{r-1}$.
\end{proof}

Again, as in previous cases we use \eqref{ro01} and \eqref{ro02} to
obtain the following result.
\begin{thm}
Let $N/n\to\infty$ and $n^2/N\to\infty$.
Then for $r=0,1$
$$
\sqrt{N}\,\frac{S_n^{(r)}-\E\,S_n^{(r)}}{n}\stackrel{d}{\to}{\cal
N}\left(0,\tilde{\sigma_r}^2\right)\;.
$$
\end{thm}

\subsection{Asymptotics in generalized inverse allocation schemes (GIAS)}\lbl{sec:gias}
Our final two settings are examples of  the gias as defined in
\eqref{eq:gias}. As in the case of GAS for  
\[S_n^{(r)}=\sum_{i=1}^N\,I(\xi_i^{(n)}=r),\] we will obtain asymptotic normality of
$$
\frac{S_n^{(r)}-\E\,S_n^{(r)}}{\sqrt{n}},
$$
when $N/n\to\lambda\in(0,\infty)$.

\subsubsection{Exchangeable negative multinomial
  model}\label{neg_mult} 
Let $(\xi_i)=(\xi_i^{(n)})$ be a random vector with a negative
multinomial distribution, that is, 
\be\label{enmm}
\P(\xi_1=k_1,\ldots,\xi_N=k_N)=\frac{(n+\sum_{j=1}^N\,k_j)!}{n!\prod_{j=1}^N\,k_j!}p^{\sum_{j=1}^N\,k_j}(1-Np)^{n+1}.
\ee
Note that this is an exchangeable case of a model  usually referred to as Bates--Neyman model,  introduced in \cite{bn}.  We refer to \cite[Chapter~36, Sec.~1--4]{jkb} for a detailed account of this distribution, its properties, applications, and further references. Here, we just note that this is a GIAS for which  \eqref{eq:gias} holds with $\eta_0\sim Poisson(\lambda_0)$,  $\eta_i\sim Poisson(\lambda_1)$, $i=1,\dots, N$, and $C=\tfrac{\lambda_0}{\lambda_0+N\lambda_1}$.  Thus \eqref{eq:gias} implies \eqref{enmm} with  $p=\frac{\lambda_1}{\lambda_0+N\lambda_1}$.

For a fixed integer $r$ we are interested in the asymptotics of
$$
S_n^{(r)}=\sum_{j=1}^N\,I(\xi_j=r).
$$
Denoting $\beta_n=(Np)^{-1}-1$ we obtain
\be\label{cc1}
c_{i,n}=(N)_i\frac{(n+ri)!}{n!(r!)^i}\,\frac{(N\beta_n)^{n+1}}{(N\beta_n+i)^{n+1+ri}}.
\ee

To study the asymptotic properties of $S_n^{(r)}$ we will assume that
$N/n\to\lambda\in(0,\infty)$. Moreover we let $p=p_n$  depend on $n$
in such a way that $Np_n\to \alpha\in(0,1)$, i.e.,  $\beta_n\to\alpha^{-1}-1$.
Consequently, setting $\Delta:=\frac{\alpha}{\lambda(1-\alpha)}$, for any $r=0,1,\ldots$
$$
\lim_{n\to\infty}\,\frac{1}{n}\,\E\,S_n^{(r)}=\frac{\lambda \Delta^r}{r!}\,e^{-\Delta}
$$
and
\be\label{s28}
\lim_{n\to\infty}\,\frac{1}{n}\,\var\,S_n^{(r)}=\frac{\lambda \Delta^r
 e^{-\Delta}}{r!}\left[1-\frac{\Delta^r
 e^{-\Delta}}{r!}(1-\lambda(r-\Delta)^2)\right]=:\sigma_r^2.
\ee

\begin{thm}
Let $N/n\to\lambda\in(0,\infty)$ and $Np_n\to \alpha\in(0,1)$.
Then for any $r=0,1,\ldots$
$$
\frac{S_n^{(r)}-\E\,S_n^{(r)}}{\sqrt{n}}\stackrel{d}{\to}{\cal
N}\left(0,\sigma_r^2\right),
$$
with $\sigma_r^2$ defined in \eqref{s28}.
\end{thm}
\begin{proof}
Write \eqref{cc1} as
$$
c_{i,n}=L_n^i\,e^{i\tfrac{n}{N\beta_n}}\,\tfrac{T(ri+1,-n)\,T(i+1,N)}{\left(1+\tfrac{i}{N\beta_n}\right)^{n+1+ri}}\quad\mbox{with}\quad L_n=\tfrac{n^re^{-\frac{n}{N\beta_n}}}{r!N^{r-1}\beta_n^r}.
$$
Thus the representation \eqref{ln} holds true with
$$
Q_{j+1}(i)=-\tfrac{i}{N\beta_n}+\left(r-\tfrac{j(n+1)}{(j+1)N\beta_n}\right)\,\left(-\tfrac{n}{N\beta_n}\right)^j\,i^{j+1}+(-1)^{j+1}\Q_{j+1}(ri+1)-\left(\tfrac{n}{N}\right)^j\Q_{j+1}(i).
$$
Moreover, $\E\,S_n^{(r)}$, $\var\,S_n^{(r)}$  and $L_n$ are all of order $n$ and thus the final conclusion follows from Remark~\ref{orden}.
\end{proof}

\subsubsection{Dirichlet negative multinomial model}\label{dir}
Finally we consider an exchangeable version of what is known as \lq
Dirichlet model of buying behaviour\rq\  introduced in a seminal paper
by Goodhardt, Ehrenberg and Chatfield \cite{gec} and subsequently
studied in numerous papers up to the present. This distribution is
also mentioned in \cite[Chapter~36, Section~6]{jkb}. Writing, as usual, $\xi_i^{(n)}=\xi_i$, the distribution under consideration has the form
$$
\P(\xi_1=k_1,\ldots,\xi_N=k_N)=\frac{\left(n+\sum_{i=1}^N\,k_i\right)!}{n!\prod_{i=1}^N\,k_i!}\,\frac{\Gamma(Na+b)}{\Gamma^N(a)\Gamma(b)}\,\frac{\Gamma(n+1+b)\prod_{i=1}^N\,\Gamma(k_i+a)}{\Gamma\left( Na+b+n+1+\sum_{i=1}^N\,k_i\right)}
$$
for any $k_i=0,1,\ldots$, $i=1,\ldots,N$. Here  $n>0$ is an integer
and $a,\,b>0$ are parameters. This is again a GIAS, for which
\eqref{eq:gias} holds with $\eta_i\sim nb(a,p)$, $i=1,\dots,N$,
$\eta_0\sim nb(b+1,p)$, $0<p<1$, and $C=\tfrac{b}{Na+b}$. When $a$ and
$b$ are integers we recall a nice interpretation of
$(\xi_1,\ldots,\xi_N)$ via the P\'olya urn scheme.  An urn contains $b$ black balls and $a$ balls in each of $N$ non-black colors. In subsequent steps a ball is drawn at random and returned to the urn together with one ball of the same color. The experiment is continued until the $n$th black ball is drawn. Then $\xi_i$ is the number of balls of the $i$th color at the end of experiment, $i=1,\ldots,N$. This distribution can also be viewed  as multivariate negative hypergeometric law of the second kind.

From the fact that $c_{i,n}=(N)_i\P(\xi_1=\ldots=\xi_i=r)$ we get
\be\label{cD}
c_{i,n}=(N)_i\,\frac{(n+ri)!}{n!(r!)^i}\,\frac{\Gamma(ia+b)}{\Gamma^i(a)\Gamma(b)}\,\frac{\Gamma(n+1+b)\Gamma^i(r+a)}{\Gamma((r+a)i+n+1+b)}.
\ee
To study the asymptotic behavior of $S_n^{(r)}$ we will assume that $N/n\to\lambda\in(0,\infty)$ and that $b=b_n$ depends on $n$ in such a way that $b_n/n\to \beta>0$.

Below we use  the following product representation of Gamma function:
\be\label{gamfun}
\Gamma(x)=\frac{1}{xe^{\gamma x}}\,\prod_{k\ge 1}\,\frac{e^{\frac{x}{k}}}{1+\frac{x}{k}},
\ee
where $\gamma$ is the Euler constant and $x>0$. We also recall (see, e.g., \cite[Section~12.16]{ww}) that for a digamma function $\Psi(x)=\frac d{dx}\ln(\Gamma(x))$ we have
\[\Psi(x+1)=-\gamma+\sum_{k\ge 1}\,\left(\frac{1}{k}-\frac{1}{k+x}\right)\quad x\ne -1,-2, \dots.\]
Then, for any $0<x<y$ we can write
\be\label{Gs}
\frac{\Gamma(x+y)}{\Gamma(y)}=e^{x\Psi(y+1)}\,\frac{y}{x+y}\,e^{\sum_{k\ge 1}\,\left(\frac{x}{k+y}-\log\left(1+\frac{x}{k+y}\right)\right)}
\ee
and the series 
\[\sum_{k\ge
  1}\,\left(\frac{x}{k+y}-\log\left(1+\frac{x}{k+y}\right)\right)\] 
converges.
Note that,
\be\label{impl}
\Psi(y)-\ln y\to0,\quad \mbox{as}\quad y\to\infty,
\ee
so that, if $\alpha_n/n\to\alpha$ then for any $x>0$
$$
n^{-x}\frac{\Gamma(\alpha_n+x)}{\Gamma(\alpha_n)}\to \alpha^x.
$$
Consequently,
$$
\lim_{n\to\infty}\,\frac{1}{n}\,\E\,S_n^{(r)}=\frac{\lambda\Gamma(a+r)}{r!\Gamma(a)}\,\frac{\beta^a}{(1+\beta)^{a+r}}.
$$
Similarly,
\be\label{s29}
\lim_{n\to\infty}\,\frac{1}{n}\,\var\,S_n^{(r)}=\frac{\lambda\Gamma(r+a)\beta^a}{r!\Gamma(a)(1+\beta)^{a+r}}
\left(1-\frac{\Gamma(r+a)\beta^a\left[1+\lambda\left(\frac{(a+r)^2}{\beta+1}-\frac{a^2}{\beta}-r^2\right)\right]}{r!\Gamma(a)(1+\beta)^{a+r}}
\right)=:\sigma_r^2.
\ee

\begin{thm}
Let $N/n\to\lambda\in(0,\infty)$ and $b_n/n\to \beta\in(0,\infty)$.
Then for any $r=0,1,\ldots$
$$
\frac{S_n^{(r)}-\E\,S_n^{(r)}}{\sqrt{n}}\stackrel{d}{\to}{\cal
N}\left(0,\sigma_r^2\right),
$$
with $\sigma_r^2$ defined in \eqref{s29}.
\end{thm}
\begin{proof}
Note that \eqref{cD} can be written as
$$
c_{i,n}=\left(\frac{N\Gamma(r+a)n^r}{r!\Gamma(a)}\right)^i\,T(ir+1,-n)\,T(i,N)\,\frac{\Gamma(b_n+ia)}{\Gamma(b_n)}
\,\frac{\Gamma(n+1+b_n)}{\Gamma(n+1+b_n+i(r+a))}.
$$
Moreover, setting $$h_j(x)=\sum_{k\ge 1}\,\frac{1}{(k+x)^{j+1}},\quad x>0,\quad j\ge2.
$$
we see that
\eqref{Gs} can be developed into
$$
\frac{\Gamma(x+y)}{\Gamma(y)}=e^{x\Psi(y+1)}\,e^{\sum_{j\ge 1}\left(\frac{(-x)^j}{jy^j}+\frac{(-x)^{j+1}}{j+1}h_j(y)\right)}.
$$
Therefore, taking $(x,y)=(ia,b_n)$ and $(x,y)=(i(r+a),n+1+b_n)$, we
decompose $c_{i,n}$ according to \eqref{ln},  where
$$
L_n=\frac{N\Gamma(r+a)n^r}{r!\Gamma(a)}\,e^{a\Psi(1+b_n)-(r+a)\Psi(2+n+b_n)}
$$
and
$$
Q_{j+1}(i)=\left[\left(\frac{n(r+a)}{b_n+n+1}\right)^j-\left(\frac{na}{b_n}\right)^j\right](-1)^{j+1}i^j$$$$+\frac{j(-n)^j}{j+1}\left[(r+a)^{j+1}h_j(b_n+n+1)-a^{j+1}h_j(b_n)\right]i^{j+1}+(-1)^{j+1}\Q_{j+1}(ir+1)-\left(\frac{n}{N}\right)^j\Q_{j+1}(i).
$$
On noting that  $\alpha_n/n\to\alpha$ implies that
$
n^jh_j(\alpha_n)<c^j(\alpha)$ uniformly in $n
$, 
we conclude that polynomials $Q_j$ satisfy condition \eqref{polbound}.
Moreover,
\eqref{impl} yields that $L_n$ is of order $n$. Since $\E\,S_n^{(r)}$ and $\var\,S_n^{(r)}$ are of order $n$ too, the result  follows by Remark \ref{orden}.
\end{proof}

\vspace{2mm}

{\bf Acknowledgements}

Part of the work of KB and JW was carried while they
    were visiting University of Seville in March/April 2009. 
 Part of the  work of PH was
carried out while he was at the Institute of Mathematics of the Polish
Academy of Sciences and the  Warsaw University of Technology in the
autumn of 2010. 
 Part of the work of GR was carried when he was visiting Warsaw University of Technology in July 2009. Part of work of JW was carried out while he was visiting Medical College of Georgia in January/February 2009.

\end{document}